\documentclass[12pt,reqno]{amsart}

\usepackage{amssymb,latexsym}

\usepackage{enumerate}

\usepackage[french,english]{babel}
\usepackage{amsmath}
\usepackage{graphicx}
\usepackage{amssymb}
\usepackage{bbm}
\usepackage{amsthm,mathtools}
\usepackage{ulem}
\usepackage{geometry}
\usepackage{tikz-cd}
\usepackage{mathrsfs}
\usepackage[colorinlistoftodos]{todonotes}
\usepackage{enumitem}
\usepackage{verbatim}
\usepackage[foot]{amsaddr}
\usepackage{dsfont}

\makeatletter

\@namedef{subjclassname@2010}{
	
	\textup{2020} Mathematics Subject Classification}

\makeatother
\newtheorem{thm}{Theorem}[section]
\newtheorem*{thm*}{Theorem}

\newtheorem{lem}[thm]{Lemma}

\theoremstyle{definition}

\numberwithin{equation}{section}

\newcommand{\mbr}{\mathbb{R}}

\newcommand{\mcf}{\mathcal{F}}

\newcommand{\mcj}{\mathcal{J}}

\newcommand{\mcm}{\mathcal{M}}

\newcommand{\mme}{\mathrm{e}}

\newcommand{\M}{{\mathcal M}}

\usepackage{hyperref}
\hypersetup{hypertex=true,colorlinks=true,linkcolor=blue,anchorcolor=blue,citecolor=blue}
\DeclareMathOperator\dif{d\!}

\newcommand{\newabstract}[1]{%
	\par\bigskip
	\csname otherlanguage*\endcsname{#1}%
	\csname captions#1\endcsname
	\item[\hskip\labelsep\scshape\abstractname.]
}

\begin{document}

	\baselineskip=17pt

	\title[Large values of Dirichlet polynomials with multiplicative coefficients]{Large values of Dirichlet polynomials with multiplicative coefficients}

	\author{Zikang Dong\textsuperscript{1}}

	\author{Yutong Song\textsuperscript{2}}
        \author{Weijia Wang\textsuperscript{3}}
            \author{Hao Zhang\textsuperscript{4}}
    \author{Shengbo Zhao\textsuperscript{2}}
	\address{1.School of Mathematical Sciences, Soochow University, Suzhou 215006, P. R. China}
	\address{2.School of Mathematical Sciences, Key Laboratory of Intelligent Computing and Applications(Ministry of Education), Tongji University, Shanghai 200092, P. R. China}
    \address{3.Morningside Center of Mathematics, Academy of Mathematics and Systems Science, Chinese Academy of
		Sciences, Beijing 100190, P. R. China}
	\address{4.School of Mathematics, Hunan University, Changsha 410082, P. R. China}
	\email{zikangdong@gmail.com}
    \email{99yutongsong@gmail.com}
	\email{weijiawang@amss.ac.cn}
	\email{zhanghaomath@hnu.edu.cn}

	\email{shengbozhao@hotmail.com}
	\date{\today}
	
	\begin{abstract} 
		In this paper, we investigate large values of  Dirichlet polynomials with multiplicative coefficients $\sum_{n\le N}f(n)n^{it}$, where $1\ll t\le T$ for large $T$. We prove an improved Omega result in the region $\exp((\log T)^{\frac12+\varepsilon})\le N\le\sqrt T$, where $T$ is large. We also show an Omega result when $\log N$ is around $\sqrt{\log T\log_2T}$.
	\end{abstract}

	
	\maketitle
	
	\section{Introduction}

Dirichlet polynomials of the form
$$
\sum_{n\le N} f(n)n^{it},
$$
where $f$ is a completely multiplicative function, play a fundamental role in analytic number theory. In the simplest case $f(n)=1$, the Dirichlet polynomial reduces to the classical zeta sum $\sum_{n\le N} n^{it}$, which has been extensively studied since the early 20th century. If $x \leq cT$ for some sufficiently small constant $c>0$, then an application of Montgomery--Vaughan's mean value theorem yields 
$$
\frac{2}{T}\int_{T/2}^{T} \Big|\sum_{n \leq x} n^{it}\Big|^2 dt 
= \lfloor x \rfloor + O\!\left(\tfrac{x^2}{T}\right) \;\gg\; x.
$$
This type of asymptotic can in fact be extended to more general ranges and for higher moments. Discussion of available bounds in this direction can be found in \cite{BS16,BMT,Harper23,Helson} and the references therein.

The resonance method, whose origin can be traced back to ideas of Voronin~\cite{Vo}, was subsequently formulated in a rigorous form and independently refined by Soundararajan~\cite{Sound} and Hilberdink~\cite{HIL}. In \cite{Sound}, Soundararajan obtained an important bound in another type 
$$
\max _{T \leqslant t \leqslant 2 T}\left|\zeta\left(\frac{1}{2}+i t\right)\right| \geqslant \exp \left((1+o(1)) \sqrt{\frac{\log T}{\log \log T}}\right).
$$
Since then, the method has been systematically developed and extended, notably by Aistleitner~\cite{Ais} and Bondarenko–Seip~\cite{BS17} among others. By constructing large GCD sums, Aistleitner used a modified version of Soundararajan's resonance method to establish lower bounds for the maximum of $|\zeta(\sigma+i t)|$ when $\sigma\in(1/2,1)$ is fixed.  The resonance method now constitutes one of the principal analytic techniques for extreme value lower bounds for $\zeta$ and general classes of $L$-functions.

Improving upon Benatar--Nishry~\cite{BN} via the resonance method,  Xu and Yang~\cite{XY} recently proved that for $\delta, \gamma \in(0,1)$ be fixed and $T=N^{C(N)}$ with $2 / \delta \leqslant C(N) \leqslant(\log N)^\gamma$, it holds for sufficiently large $N$ that
$$
\max _{\frac{T}{2} \leqslant t \leqslant T}\left| \sum_{n \leqslant N} f(n) n^{i t}\right| \geqslant \sqrt{N}\exp \left(\sqrt{(1-\delta) \frac{\log T}{\log \log T}}\right),
$$
uniformly for all completely multiplicative functions $f$ with $|f(n)|=1$.

In this paper, we focus on large values in ranges of $N$ between $\exp((\log T)^{1/2+\varepsilon})$ and $\sqrt{T}$, as well as in transition regimes where $\log N = \sqrt{\log T\log_2T}(\log_2T)^{O(1)}$.  
    
\begin{thm}\label{thm1.1}
    Let $\log N=(\log T\log_2T)^{\frac12}\tau$ with $\tau=(\log_2T)^{O(1)}$. Let $f$  be any completely multiplicative function such that $|f(n)|=1$ for any integer $n$. Then we have
    $$\max_{1\leqslant t\leqslant T}\Big|\sum_{n\le N}f(n)n^{it}\Big|\ge \sqrt{N}\exp\bigg((1+o(1))A(\tau+\tau'){\sqrt{\frac{\log T}{\log_2 T}}}\bigg).$$
    where $A,\tau'\in\mbr$ such that 
   \[\tau=\int_A^\infty \frac{\mme^{-u}}{u}d u,\qquad \tau'=\int_A^\infty\frac{\mme^{-u}}{u^2}d u.\]
\end{thm}
In particular, if we choose $f(\cdot)=1$ always, then Theorem \ref{thm1.1} overcomes Theorem 1.6 of \cite{DWZ}.

For any $c>0$, we define a subset of the above completely multiplicative functions.
\begin{equation*}
    \mcf(c) \coloneqq \big\{f\;{\rm completely\; multiplicative}:|f(n)|=1, \forall n \in \mathbb{N};\, {\rm{Re}}f(n)\overline{f(m})\ge c,\;\forall \,m,n \big\}.
\end{equation*}

Combining with GCD sum estimates, we deduce the following $\Omega$-type lower bounds for Dirichlet polynomials with coefficients in $\mcf(c)$.

\begin{thm}\label{thm1.2}
Fix $\delta \in (0,\frac{1}{100})$ and $c>0$. Let $T$ be a sufficiently large number and $N$ be any positive number satisfying $\exp((\log T)^{\frac{1}{2}+\delta}) \le N \le T^{\frac{1}{2}}$. Then we have
    $$ \max_{1\leqslant t\leqslant T}\Big|\sum_{n\le N}f(n)n^{it}\Big|\ge \sqrt{N}\exp\bigg((\sqrt2+o(1)){\sqrt{\frac{\log (T/N)\log_3(T/N)}{\log_2 (T/N)}}}\bigg)$$ 
    for all $f \in \mcf(c)$.
\end{thm}
This improves the result of Xu and Yang \cite{XY} at the cost of an additional request for $f(\cdot)$ and a wider range of $t$.
\section{Proof of Theorem \ref{thm1.1}}
Let $x=T/N$ and $\lambda=\sqrt{\log x\log_2x}$. We  define the completely multiplicative function $r(n)$ by $$r(p)=\begin{cases}
   \frac{\lambda}{\sqrt p \log p}, &  \lambda^2\le p\le \exp((\log\lambda)^2),\\
   0, & {\rm otherwise.}
   \end{cases}$$ We define the resonator $R(t)$ by 
\[R(t):=\sum_{n\leq x}f(n)r(n)n^{it}.\]
Let  $\Phi(t) := e^{-t^2 /2}$. We have the Fourier transform $\widehat{\Phi}(x) = \sqrt{2\pi}\Phi(x) > 0$.
Write $S_t(N) \coloneqq S_{f,t}(N) = \sum_{n \le N}f(n)n^{\mathrm{i}t}
$.
Define 
$$M_1(R,T):=\frac{\log T}{T}\int_{1\le|t|\le T} |R(t)|^2 \Phi\Big( \frac{t \log T}{T} \Big)\dif t,$$
and
$$M_2(R,T):=\frac{\log T}{T}\int_{1\le|t|\le T} S_t(N)|R(t)|^2 \Phi\Big( \frac{t \log T}{T} \Big)\dif t.$$
Then we have 
\[\max_{t\in [1,T]}|S_t(N)|\geq \frac{|M_2(R,T)|}{M_(R,T)}.\]

For $M_1$ we have
\begin{align*}M_1(R,T)&\le I_1(R,T)\\&:=\frac{\log T}{T}\int_{\mathbb R} |R(t)|^2 \Phi\Big( \frac{t \log T}{T} \Big)\dif t\\&=\sum_{m,n\le x}f(m)\overline f(n)r(m)r(n)\widehat\Phi\Big(\frac{T}{\log T}\log\frac mn\Big)\\
&=\sum_{m\le x}r(m)^2+\sum_{m,n\le x\atop m\neq n}f(m)\overline f(n)r(m)r(n)\widehat\Phi\Big(\frac{T}{\log T}\log\frac mn\Big).\end{align*}
When $m\neq n$, since $m,n\le x$, we have for some $c>0$
$$\big|\log \frac mn\big|\ge\log\big(1+\frac 1x\big)\ge\frac cx.$$
So \begin{align*}\widehat\Phi\Big(\frac{T}{\log T}\log\frac mn\Big)\le\widehat\Phi\Big(\frac{T}{\log T}\frac  cx\Big)=\widehat\Phi\Big(\frac{T}{\log T}\frac {cN}T\Big)=\widehat\Phi\Big(\frac{cN}{\log T}\Big).\end{align*}
Thus we have
\begin{align*}&\sum_{m,n\le x\atop m\neq n}f(m)\overline f(n)r(m)r(n)\widehat\Phi\Big(\frac{T}{\log T}\log\frac mn\Big)\\&\le\widehat\Phi\Big(\frac{cN}{\log T}\Big)\sum_{m,n\le x}r(m)r(n)\\
&\le x^2\widehat\Phi\Big(\frac{cN}{\log T}\Big)\sum_{m\le x}r(m)^2\\
&=o\Big(\sum_{m\le x}r(m)^2\Big).\end{align*}
So we finally have for $M_1$
$$M_1(R,T)\le(1+o(1))\sum_{m\le x}r(m)^2.$$

For $M_2$, firstly we have
\begin{align*}\int_{|t|\le 1} S_t(N)|R(t)|^2 \Phi\Big( \frac{t \log T}{T} \Big)\dif t&\le\sum_{k\le N}\sum_{m,n\le x}r(m)r(n)\\&\le Nx\sum_{m\le x}r(m)^2\\&=T\sum_{m\le x}r(m)^2,\end{align*}
and 
$$\int_{|t|\ge T} S_t(N)|R(t)|^2 \Phi\Big( \frac{t \log T}{T} \Big)\dif t\ll1.$$
So 
$$M_2(R,T)=I_2(R,T)+O\Big(\log T\sum_{m\le x}r(m)^2\Big),$$
where
$$I_2(R,T):=\int_{\mathbb R} S_t(N)|R(t)|^2 \Phi\Big( \frac{t \log T}{T} \Big)\dif t.$$
Expanding it, we have
\begin{align*}
 I_2(R,T)&=\sum_{k\le N}\sum_{m,n\le x}f(km)\overline f(n)r(m)r(n) \widehat\Phi\Big(\frac{T}{\log T}\frac{km}{n}\Big)\\
 &=\sum_{k\le N}\sum_{m,n\le x\atop km=n}r(m)r(n)+\sum_{k\le N}\sum_{m,n\le x\atop km\neq n}f(km)\overline f(n)r(m)r(n) \widehat\Phi\Big(\frac{T}{\log T}\frac{km}{n}\Big)\\
 &=\sum_{k\le N}\sum_{m\le x}r(m)r(km)+\sum_{k\le N}\sum_{m,n\le x\atop km\neq n}f(km)\overline f(n)r(m)r(n) \widehat\Phi\Big(\frac{T}{\log T}\frac{km}{n}\Big).  
\end{align*}
Similarly to our treatment for $I_1(R,T)$, we  have
$$I_2(R,T)=(1+o(1))\sum_{k\le N}\sum_{m\le x}r(m)r(km)=(1+o(1))\sum_{k\le N}r(k)\sum_{m\le x/k}r(m)^2.$$
For the inner sum, following the step of Hough's work \cite[p. 105]{Hou}, we have
$$\sum_{m\le x/k}r(m)^2=(1+o(1))\sum_{m\ge1}r(m)^2\ge(1+o(1))\sum_{m\le x}r(m)^2.$$
Now we obtain 
$$\max_{1\le t\le T}|S_t(N)|\ge \frac{M_2(R,T)}{M_1(R,T)}\ge\frac{I_2(R,T)}{I_1(R,T)}+O(\log T)\ge(1+o(1)){\sum_{k\le N}r(k)}+O(\log T).$$
Now Theorem \ref{thm1.1} follows from  \cite[p.105--107]{Hou} immediately.

\section{Proof of Theorem \ref{thm1.2}}

Firstly, we need a lemma that maximizes the GCD sum.
\begin{lem}\label{lemmaGCD}
    If $\mcm$ is any set of positive integers satisfying $\max \mcm \le 2 \min \mcm$ and $|\mcm| = K$ is large, then
    $$
    \max_{|\mcm| = K}\frac{1}{|\mcm|} \sum_{m, n \in \mcm}\sqrt{\frac{(m,n)}{[m,n]}} = \exp \bigg( (2\sqrt{2}+o(1))\sqrt{\frac{\log K \log_3 K}{\log_2 K
    }} \bigg).
    $$
\end{lem}
\begin{proof}
    This is Theorem 1.1 of \cite{BT}.
\end{proof}

    Note that in the proof of the above lemma, the choice for the set $\M$ satisfies $y_\M:=\max_{m\in\M}P(m)\le (\log K)^{1+o(1)}$.
    
Let $\mcm$ be a set of positive integers satisfying the conditions in Lemma \ref{lemmaGCD} with cardinality $|\mcm|  = \lfloor T/N \rfloor$.
For $j \ge 0$, define
$$
\mcm_j \coloneqq \Big[ \Big(1 + \frac{\log T}{T}\Big)^j, \Big(1 + \frac{\log T}{T}\Big)^{j+1} \Big) \bigcap \mcm.
$$
Let $\mcj$ be the set of integers $j$ such that $\mcm_j \neq \varnothing$, and let 
$$
\mcm^{\prime} \coloneqq \{ m_j = \min \mcm_j : j \in \mcj \}.
$$
Then we define the resonator as follows
$$
R(t) \coloneqq \sum_{m \in \mcm^{\prime}} f(m)r(m)m^{- \mathrm{i}t} ,
$$
where $r: \mcm^{\prime} \to \mathbb{R}^+$ and its value at $m_j$ is $r(m_j) = |\mcm_j|^{1/2}$. Plainly, the Cauchy-Schwarz inequality implies that
\begin{equation}\label{RupperboundN}
    |R(t)| \le \sum_{m \in \mcm^{\prime}}r(m) \le \Big( \sum_{m \in \mcm^{\prime}} 1 \Big)^{\frac{1}{2}} \Big( \sum_{m \in \mcm^{\prime}} r(m)^2 \Big)^{\frac{1}{2}} \leq |\mcm| .
\end{equation}
\par
Let  $\Phi(t) \coloneqq e^{-t^2 /2}$. We have the Fourier transform $\widehat{\Phi}(x) = \sqrt{2\pi}\Phi(x) > 0$. Define the following two moments
$$
 M_1(R,T) \coloneqq \int_{1 \le |t| \le T} |R(t)|^2 \Phi\Big( \frac{t \log T}{T} \Big) \mathrm{d}t
$$
and
$$
 M_2(R,T) \coloneqq \int_{1 \le |t| \le T} |S_t(x)|^2|R(t)|^2 \Phi\Big( \frac{t \log T}{T} \Big) \mathrm{d}t,
$$
where $S_t(N) \coloneqq S_{f,t}(N) = \sum_{n \le N}f(n)n^{\mathrm{i}t}$.

Obviously we have
\begin{equation}\label{max1}
    \max_{1 \le |t| \le T} |S_t(N)|^2 \geq \frac{M_2(R, T)}{M_1(R, T)}.
\end{equation}
\par
For $M_1(R,T)$, we have the following bound
\begin{equation}
    \label{I1upperbouund}
    M_1(R,T)\le I_1(R,T) \coloneqq \int_{\mathbb{R}} |R(t)|^2 \Phi\left( \frac{t \log T}{T} \right) \mathrm{d}t \ll \frac{T |\mcm|}{\log T}.
\end{equation}
The proof for the last inequality is similar to \cite[Lemma 5]{BS2}. For the sake of completeness, we present the detailed process below. To see this, we note that 
$$
I_1(R,T) = \sqrt{2\pi}\frac{T}{\log T}\sum_{m,n \in \mcm^{\prime}}f(m)\overline{f(n)}r(m)r(n)\Phi\Big( \frac{T}{\log T}\log \frac{m}{n}\Big).
$$
For $m=n$, employing the definition of $r(\cdot)$ and the fact that $|f(n)|=1$ we have that
$$
\sum_{m \in \mcm^{\prime}}r(m)^2 = \sum_{j \in \mcj}r(m_j)^2 = \sum_{j \in \mcj}|\mcm_j| = |\mcm|.
$$
For $m \neq n$ we have
\begin{align*}
   &\sum_{\substack{m,n \in \mcm^{\prime} \\ m \neq n}}  f(m)\overline{f(n)}r(m)r(n)\Phi\Big( \frac{T}{\log T}\log \frac{m}{n}\Big)  \\
 & \le\sum_{\substack{m,n \in \mcm^{\prime} \\ m \neq n}}  r(m)r(n)\Phi\Big( \frac{T}{\log T}\log \frac{m}{n}\Big) 
\end{align*}
By the definition of $\Phi(t)$ an $\M'$,  the above is
\begin{align*}
      & \le\sum_{\substack{j,k \in \mcj \\ j \neq k}}r(m_j)r(m_k)\Phi\Big(\frac{T}{\log T}(|j-k|-1)\log \Big(1+\frac{\log T}{T} \Big) \Big) \\
    &\le \sum_{\substack{j,k \in \mcj \\ j \neq k}}r(m_j)r(m_k)\Phi(|j-k|-1) .
\end{align*}
Since $\Phi(t)$ decays rapidly we have this is
\begin{align*}
   &\ll \sum_{\ell \ge1}\Phi(\ell-1) \sum_{\substack{j,k \in \mcj \\ |j - k| = \ell}} r(m_j)r(m_k)\\
     &\ll\sum_{\ell \ge1}\Phi(\ell-1) \sum_{j \in \mcj } r(m_j)r(m_{j+\ell})\\
     &\ll \sum_{j \in \mcj } r(m_j)^2\sum_{\ell \ge1}\Phi(\ell-1)\\
    &\ll|\M|.
\end{align*}
Thus we get \eqref{I1upperbouund}.
\par

Now we focus on $M_2(R,T)$. We have the following estimate by the rapid decay of $\Phi(t)$, we have
\begin{equation}\label{tlargeupper}
    \int_{|t| > T} |S_t(x)|^2 |R(t)|^2 \Phi\Big( \frac{t \log T}{T} \Big) \mathrm{d}t \ll 1.
\end{equation}
Furthermore, \eqref{RupperboundN} yields
\begin{equation}\label{tsmallupper}
    \int_{|t| < 1} |S_t(N)|^2 |R(t)|^2 \Phi\Big( \frac{t \log T}{T} \Big) \mathrm{d}t \le N^2 \Big|\sum_{m\in\M'}r(m)\Big|^2 \leq  N^2|\M'||\M| \leq T^2.
\end{equation}
Combining \eqref{tlargeupper} and \eqref{tsmallupper}, we have 
\begin{equation}\label{I2equation}
    I_2(R, T) \coloneqq \int_{\mathbb{R}} |S_t(N)|^2 |R(t)|^2 \Phi\Big( \frac{t \log T}{T} \Big) \mathrm{d}t =M_2(R, T) + O(T^2).
\end{equation}
Plugging \eqref{I1upperbouund} and \eqref{I2equation} into \eqref{max1}, we obtain
\begin{equation}\label{max2}
    \max_{1 \le |t| \le T} |S_t(N)|^2 \geq \frac{I_2(R, T)+O(T^2)}{I_1(R, T)}  \gg \frac{\log T}{T |\mcm|}I_2(R, T)+O(N\log T).
\end{equation}
\par
To deal with $ I_2(R, T)$, we expand it according to the definitions of $R(t)$ and $S_t(N)$ and obtain 
\begin{align*}
    I_2(R, T) = \, & \sum_{a,b \leq N}\sum_{m, n \in \mcm^{\prime}} f(am)\overline{f(bn)} r(m)r(n) \int_{\mathbb{R}} \Big( \frac{ma}{nb} \Big)^{-\mathrm{i}t} \Phi\Big( \frac{t \log T}{T} \Big) \mathrm{d}t \nonumber \\
    = \, & \frac{T}{\log T} \sum_{a,b \leq N}  \sum_{m, n \in \mcm^{\prime}} f(am)\overline{f(bn)} r(m) r(n) \widehat{\Phi}\Big( \frac{T}{\log T} \log \frac{ma}{nb} \Big). 
\end{align*}
We split the sum into two parts: the case $a m = b n$ and the case $a m \neq b n$,
\begin{align}\label{I2eq}
    I_2(R, T) = \, & \frac{T}{\log T} \sum_{a,b \leq N} \Big(\sum_{\substack{m, n \in \mcm^{\prime} \\ am=bn}} + \sum_{\substack{m, n \in \mcm^{\prime} \\ am \neq bn}} \Big)f(am)\overline{f(bn)} r(m) r(n) \widehat{\Phi}\Big( \frac{T}{\log T} \log \frac{ma}{nb} \Big) \nonumber \\
          = \, & \frac{T}{\log T}\widehat{\Phi}(0)\sum_{a,b \leq N}\sum_{\substack{m, n \in \mcm^{\prime} \\ am=bn}}r(m)r(n) \nonumber \\
          & + \, \frac{2T}{\log T} {\rm Re}\sum_{a,b \leq N}\sum_{\substack{m, n \in \mcm^{\prime} \\ am < bn}} f(am)\overline{f(bn)} r(m) r(n) \widehat{\Phi}\Big( \frac{T}{\log T} \log \frac{ma}{nb} \Big) \nonumber \\
          \,\ge & \frac{T}{\log T}\widehat{\Phi}(0)\sum_{a,b \leq x}\sum_{\substack{m, n \in \mcm^{\prime} \\ am=bn}}r(m)r(n) \nonumber \\
          & + \, c\frac{2T}{\log T} \sum_{a,b \leq N}\sum_{\substack{m, n \in \mcm^{\prime} \\ am < bn}}  r(m) r(n) \widehat{\Phi}\Big( \frac{T}{\log T} \log \frac{ma}{nb} \Big) \nonumber \\          
          \gg \, & \frac{T}{\log T}\sum_{a,b \leq N}\sum_{m, n \in \mcm^{\prime} }r(m)r(n)\widehat{\Phi}\Big( \frac{T}{\log T} \log \frac{ma}{nb} \Big).
\end{align}

For any fixed  $a,b\le N$, we have for the inner sum
\begin{align*}
&\sum_{m, n \in \mcm^{\prime} }r(m)r(n)\widehat{\Phi}\Big( \frac{T}{\log T} \log \frac{ma}{nb} \Big)\\
&=\sum_{j, k\in \mcj }r(m_j)r(m_k)\widehat{\Phi}\Big( \frac{T}{\log T} \log \frac{m_ja}{m_kb} \Big)\\
&\ge\sum_{j, k\in \mcj  }\min\{r(m_j)^2,r(m_k)^2\}\widehat{\Phi}\Big( \frac{T}{\log T} \log \frac{m_ja}{m_kb} \Big)\\
&\ge\sum_{j, k\in \mcj }\sum_{n_j\in\M_j,n_k\in\M_k\atop an_j=bn_k}\widehat{\Phi}\Big( \frac{T}{\log T} \log \frac{m_ja}{m_kb} \Big).
\end{align*}
By the definition of $\M_j$ and $\M_k$, we have 
$$\widehat{\Phi}\Big( \frac{T}{\log T} \log \frac{m_ja}{m_kb} \Big)\gg\widehat{\Phi}\Big( \frac{T}{\log T} \log \frac{n_ja}{n_kb} \Big).$$
So the above  is
\begin{align*}
&\gg\sum_{j, k\in \mcj }\sum_{n_j\in\M_j,n_k\in\M_k\atop an_j=bn_k}\widehat{\Phi}\Big( \frac{T}{\log T} \log \frac{n_ja}{n_kb} \Big)\\
&=\sum_{j, k\in \mcj }\sum_{n_j\in\M_j,n_k\in\M_k\atop an_j=bn_k}1\\
&=\sum_{m,n\in\M\atop am=bn}1.
\end{align*}
Back to \eqref{I2eq}, we have
\begin{align*}
    I_2(R, T)&\gg \frac{T}{\log T}\sum_{a,b \leq N}\sum_{m,n\in\M\atop am=bn}1\\
    &=\frac{T}{\log T}\sum_{m,n\in\M}\sum_{a,b \leq N\atop am=bn}1\\
    &\ge\frac{T}{\log T}\sum_{m,n\in\M\atop{m}/{(m,n)},{n}/{(m,n)}\le N}\frac{N}{\max\{\frac{m}{(m,n)},\frac{n}{(m,n)}\}}\\
    &\ge\frac{T}{\log T}\sum_{m,n\in\M\atop{m}/{(m,n)},{n}/{(m,n)}\le N}\frac{N}{\sqrt{2\frac{m}{(m,n)}\frac{n}{(m,n)}}}\\
    &\gg\frac{NT}{\log T}\sum_{m,n\in\M\atop{{[m,n]}}/(m,n)\le N^2/2}{\sqrt{\frac{(m,n)}{[m,n]}}}.
\end{align*}
For the inner GCD sum, we have
\begin{align*}
\sum_{m,n\in\M\atop{{[m,n]}}/(m,n)\le N^2/2}{\sqrt{\frac{(m,n)}{[m,n]}}}
&=\bigg(\sum_{m,n\in\M
}\sqrt{\frac{(m,n)}{[m,n]}}-\sum_{m,n\in\M
\atop [m,n]/(m,n)> N^2/2}\sqrt{\frac{(m,n)}{[m,n]}}\bigg)\\
&\gg|\M|\exp\bigg((2\sqrt2+o(1))\sqrt{\frac{\log(T/N)\log_3(T/N)}{\log_2(T/N)}}\bigg),
\end{align*}
since
\begin{align*}\sum_{m,n\in\M
\atop [m,n]/(m,n)> N^2/2}\sqrt{\frac{(m,n)}{[m,n]}}&\ll N^{-2\eta}\sum_{m,n\in\M
}\bigg({\frac{(m,n)}{[m,n]}}\bigg)^{\frac12-\eta}\\
&\ll N^{-2\eta}\prod_{p\le y_\M}\bigg(1+\frac{2}{p^{\frac12-\eta}-1}\bigg)\\
&\ll N^{-2\eta}\exp\big( y_\M^{\frac12+\eta}\big)\\
&\ll N^{-2\eta}\exp\big( (\log (T/N))^{\frac12+\eta+o(1)}\big)\\
&\ll \exp\big(-\tfrac23\delta(\log T)^{\frac12+\delta}\big)\exp\big((\log T)^{\frac12+\frac{2}{3}\delta}\big)\\
&\ll1,
\end{align*}
with $\eta=\delta/3$, $y_\M:=\max_{m\in\M} P(m)\le(\log (T/N))^{1+o(1)}$ and $N>\exp((\log T)^{\frac12+\delta}).$
So we  get
$$I_2(R, T)\gg \frac{NT}{\log T}|\M|\exp\bigg((2\sqrt2+o(1))\sqrt{\frac{\log(T/N)\log_3(T/N)}{\log_2(T/N)}}\bigg).$$
And finally by \eqref{max2} we have
 \begin{align*}
 &\max_{1 \le |t| \le T} |S_t(N)|^2\\&\gg \frac{\log T}{T |\mcm|} \frac{NT}{\log T}|\M|\exp\bigg((2\sqrt2+o(1))\sqrt{\frac{\log(T/N)\log_3(T/N)}{\log_2(T/N)}}\bigg)+O(N\log T)\\
 &=N\exp\bigg((2\sqrt2+o(1))\sqrt{\frac{\log(T/N)\log_3(T/N)}{\log_2(T/N)}}\bigg).
 \end{align*}
Thus we complete the proof.
	\section*{Acknowledgements}
	Z. Dong is supported by the Shanghai Magnolia Talent Plan Pujiang Project (Grant No. 24PJD140) and the National
	Natural Science Foundation of China (Grant No. 	1240011770). W. Wang is supported by the National
	Natural Science Foundation of China (Grant No. 12500). H. Zhang is supported by the Fundamental Research Funds for the Central Universities (Grant No. 531118010622), the National
	Natural Science Foundation of China (Grant No. 1240011979) and the Hunan Provincial Natural Science Foundation of China (Grant No. 2024JJ6120).
	
	\normalem

\end{document}